\documentclass[10pt]{amsart}
\usepackage{amsmath}
\usepackage{amsfonts}
\usepackage{amssymb}
\usepackage{graphicx}
\usepackage{amsthm,graphicx,color,yfonts}
\usepackage{pdfsync}
\usepackage{epstopdf}
\usepackage{epsfig}

\usepackage[colorlinks=true]{hyperref}
\hypersetup{linkcolor=red,citecolor=blue,filecolor=dullmagenta,urlcolor=darkblue} 

\numberwithin{equation}{section}

\theoremstyle{plain}
\newtheorem{theorem}{Theorem}[section]

\newtheorem{step}{Step}
\newtheorem{lemma}[theorem]{Lemma}
\newtheorem*{de-lemma}{Lemma}

\newtheorem{corollary}[theorem]{Corollary}

\theoremstyle{remark}

\theoremstyle{definition}

\DeclareMathOperator{\supp}{supp}

\newcommand{\dd}{\mathrm{d}}
\newcommand{\R}{\mathbb{R}}

\newcommand{\ve}{\epsilon}

\begin{document}

\title[Phase transition model]{Theory of light-matter interaction in nematic liquid crystals and the second Painlev\'e equation}

\author{Marcel G. Clerc}
\address{Departamento de F\'{i}sica, FCFM, Universidad de Chile, Casilla 487-3, Santiago,
Chile.}
\email{marcelclerc@gmail.com}
\thanks{M.G. Clerc was partially supported by Fondecyt 1150507.}

\author{Juan Diego D\'avila}
\address{Departamento de Ingenier\'{\i}a Matem\'atica and Centro
de Modelamiento Matem\'atico (UMI 2807 CNRS), Universidad de Chile, Casilla
170 Correo 3, Santiago, Chile.} 
\email{jdavila@dim.uchile.cl}
\thanks{J.D\'avila was partially supported by Fondecyt 1130360 and Fondo Basal CMM-Chile}

\author{Micha{\l } Kowalczyk}
\address{Departamento de Ingenier\'{\i}a Matem\'atica and Centro
de Modelamiento Matem\'atico (UMI 2807 CNRS), Universidad de Chile, Casilla
170 Correo 3, Santiago, Chile.}
\email {kowalczy@dim.uchile.cl}
\thanks{M. Kowalczyk was partially supported by Chilean research grants Fondecyt 1130126, Fondo Basal CMM-Chile}

\author{Panayotis Smyrnelis}
\address{Centro
de Modelamiento Matem\'atico (UMI 2807 CNRS), Universidad de Chile, Casilla
170 Correo 3, Santiago, Chile.}
\email{psmyrnelis@dim.uchile.cl}
\thanks{P. Smyrnelis was partially supported by Fondo Basal CMM-Chile and Fondecyt postdoctoral grant 3160055}

\author{Estefania Vidal-Henriquez}
\address{Max Planck Institute for Dynamics and Self-Organization, 
Am Fassberg 17, D-37077 G{\"o}ttingen, Germany.}
\email{tefavidal@gmail.com}
\thanks{E. Vidal-Henriquez was partially supported by a  Master fellowship  CONICYT  
221320023 and  DPP of the University of Chile}


\begin{abstract}
We study global minimizers of an energy  functional arising as a thin sample limit in the theory of light-matter interaction in nematic liquid crystals. We show that depending on the parameters various defects are predicted by the model. In particular we show existence of a new type of topological  defect which we call the {\it shadow kink}. 
Its local profile is described by the second Painlev\'e equation. As part of our analysis we find new solutions to this equation thus generalizing  the well known result of Hastings and McLeod \cite{MR555581}.  
\end{abstract}

\maketitle

\section{Introduction}

\subsection{Physical motivation}

\label{sec:motivation}

In a suitable experimental set up \cite{Barboza2012,Barboza2015, Barboza2015A, LCLV-Vortex2015,PhysRevLett.111.093902, clerc2} involving a liquid crystal sample, a laser and a photoconducting cell one can observe light defects  such as kinks, domain walls and vortices. A concrete example of formation of optical vortices is presented in \cite{clerc2}.

To describe the energy of  the illuminated liquid crystal light valve (LCLV) filled with a negative dielectric 
nematic liquid crystal which is  homeotropically  anchored, we consider  the Oseen-Frank model in the vicinity of the
Fr\'{e}edericksz transition.
Denoting the molecular director by $\vec{n}$ the Oseen-Frank  energy is given by 
\cite{DeGennes}
\begin{equation}
\label{1.4}
\mathcal F=\int \frac{K_1}{2} (\nabla \cdot \vec n)^2+\frac{K_2}{2} \left(\vec n \cdot(\nabla \times \vec n)\right)^2+\frac{K_3}{2} \left(\vec n\times (\nabla\times \vec n)\right)^2-\frac{\varepsilon_a}{2} (\vec E\cdot \vec n)^2.
\end{equation}
where $\{K_1,K_2,K_3\}$ are, respectively,  the splay, twist, and bend  elastic constants of the nematic liquid crystal and  $\varepsilon _{a}$ anisotropic dielectric constant ($\varepsilon _{a}<0$). We will neglect the anisotropy i.e we will assume that $K_1=K_2=K_3=K$. 
Under uniform illumination $\vec E= [V_0+a I]/d$ $\hat z $, 
where $V_0$ is the voltage applied to the LCLV,  $d$ thickness of the cell, $I$ intensity of the illuminating light beam, and
$a$ is a phenomenological dimensional parameter
that describes the linear response of the photosensitive wall \cite{ResidoriReport}.
The homeotropic state, $\vec{n}=\hat{z}$, undergoes a stationary instability
for critical values of the voltage which match the Fr\'eedericksz transition
threshold $V_{FT}=\sqrt{-K\pi ^{2}/\varepsilon _{a}}-a I$. 

Illuminating the liquid crystal light valve with a 
Gaussian beam induces a voltage drop with a bell-shaped profile
across the liquid crystal layer, higher in the center 
of the illuminated area. The electric field within  the thin sample   takes the form \cite{Barboza2015}
\begin{equation}
\vec{E} = E_z\hat{z}+ E_r \hat{r} \equiv\dfrac{\left[V_0+a I(r) \right]}{d} \hat{z} 
+\dfrac{z a}{d\omega} I'(r) \hat{r}, 
\end{equation}
where $r$ is  the radial coordinate centered on the beam, $\hat{r}$ the unitary radial vector,
$I(r)$  the intensity of Gaussian light beam, $I(r)=I_0e^{-r^2/2\omega^2}$,
$I_0$  the peak intensity, and $\omega$ the width of the light beam.

If  the intensity of the light beam is sufficiently  close to the Fr\'eedericksz transition the director  is slightly tilted from the $\hat z$ direction and 
one can use the following ansatz 
\begin{equation}
\label{ans_n}
\vec{n} (x,y,z)\approx\left(
\begin{array}
[c]{c}
n_1(x,y,\pi z/ d)\\
n_2(x,y,\pi z/ d)\\
1-\frac{(n_1^{2}+n_2^{2})}{2}
\end{array}
\right).
\end{equation}
Introducing the above ansatz  in the energy functional  $\mathcal F$ and taking the limit of the thickness of the sample $d\to 0$  
one obtains  the following problem (written here for simplicity in  a non dimensional form)
 \cite{Frisch_2, Barboza2012,Barboza2015} 
%
\begin{equation}
\label{E-FGL_1}
G(u)=\int_{\R^2} \frac{\ve}{2}|\nabla u|^2-\frac{1}{2\ve} \mu(x,y) |u|^2+\frac{1}{4\ve} |u|^4-a \left(f_1(x,y) u_1+f_2(x,y)u_2\right),
\end{equation}
where $u=(u_1, u_2)\colon \R^2\to \R^2$ is an order parameter describing the tilt of $\vec{n}$ from the $\hat z$ direction  in the thin sample limit, $\epsilon\ll 1$ is  proportional to  $K$ and in radial co-ordinates  
\begin{equation}\label{Pas}
\mu(x,y)= e^{\,-r^2}-\chi, \qquad f(x,y)=-\frac{1}{2}e^{\,i\theta}\frac{d}{dr}[e^{\,-r^2}-\chi]= e^{\,i\theta} re^{\,-r^2}, \qquad (x,y)=re^{\,i\theta},
\end{equation}
and $\chi\in (0,1)$  is a fixed constant. 
The function $\mu$ describes light intensity and is sign changing due  to the fact that the light 
is applied to the sample locally and  areas where $\mu<0$ are interpreted as  shadow zones 
while  areas where $\mu>0$ correspond to  illuminated zones.  The function $f$ describes 
the electric field induced by the light due to the photo conducting {{blue}wall} mounted 
on top of the sample \cite{LCLV-Vortex2015}. Experiments show that as the intensity of the 
applied laser  light represented here explicitly by the parameter $a$ increases, 
defects such as light vortices appear first on the border of the illuminated zone and then in its center. 
This {transition} takes places suddenly once a threshold value of $a$ is attained. 
At  large values of $a$ vortices  have local profiles resembling the profile of the standard 
vortex of degree {$+1$} in the Ginzburg-Landau theory. At low values of $a$ vortices 
are located in the shadow area (we call them shadow vortices) and their local profiles are 
very different than that of the standard ones. In particular while the amplitude of the 
standard vortex is of order $\mathcal O(1)$ in $\ve$ the amplitude of the shadow vortex  
is of order $\mathcal O(\ve^{1/3})$. This picture is confirmed experimentally, numerically 
and by formal calculations \cite{clerc2}. Currently new experiments are being designed in order to realize experimentally other types of defects, such as kinks or domain walls. In the context of the model energy (\ref{E-FGL_1}) this amounts to assuming that $u_2\equiv 0$ (domain walls) or $u=u(x)$ and $u_2\equiv 0$ (kinks). In the latter case the energy takes form
\begin{equation}
\label{funct 00}
E(u)=\int_{\R}\frac{\epsilon}{2}|u_x|^2-\frac{1}{2\epsilon}\mu(x)u^2+\frac{1}{4\epsilon}|u|^4-a f(x)u,
\end{equation}
with $\mu(x)$ and $f(x)$ given by:
\begin{equation}
\label{phys relev}
\mu(x)=e^{\,-x^2}-\chi, \quad \chi\in (0,1), \qquad f(x)=-\frac{1}{2}\mu'(x)=x e^{\,-x^2} ,
\end{equation}
where $\chi\in (0, 1)$ is fixed. 



In this paper we will study  global  minima of the problem (\ref{funct 00}). The energy $E(u)$  is a real valued, one dimensional version of   
$G(u)$, yet both show a remarkable qualitative agreement. This is not surprising in view of the fact that both of them come from taking the thin sample limit of the Oseen-Frank energy (\ref{1.4}).  {The} theoretical value of our study lies in understanding and explaining  the basic mechanism of formation of the  various types of defects  on the basis of the analogous mechanism for the the energy $E(u)$. In particular we will show  existence of a new type of defect,  the shadow kink,  appearing at the points where $\mu$ changes sign i.e. in the shadow area of the one dimensional model. Its analog for the energy $G$ is the shadow vortex \cite{clerc2} and here we make a first step in understanding its local profile via the second Painlev\'e equation.

The model of light-matter interaction  in nematic liquid crystals described above  has some similarities with the  model of the Bose-Einstein condensates in a rotating trap   based on  the Gross-Pitaevskii energy
\[
F(u)=\int_{\R^2} \frac{1}{2}|\nabla u|^2 +\frac{1}{2\ve^2} V(x)|u|^2+\frac{1}{4\ve} |u|^4- \Omega x^\perp\cdot  (iu, \nabla u) \quad \mbox{subject to}\quad \|u\|_{L^2}=1 ,
\] 
where $\Omega\in \R$ is the angular velocity,  $ (iu, \nabla u)=i u\nabla \bar u -i\bar u \nabla u$ and $V(x)= x_1 +\Lambda x_2$ is a harmonic trapping potential (more general  nonnegative, smooth $V$ are considered as well).  
The role played in $G(u)$ or $E(u)$ by the parameter $a$ is played here by the angular velocity, whose  threshold values correspond to emergence of  global minimizers of different nature. 
When  $\Omega=\mathcal O(|\ln \ve|)$ is below a critical value $\Omega_1$  
global minimizers are   vortex free \cite{Ignat2006260,MR2772375}, while at some 
other critical values $\Omega_2>\Omega_1$ global minimizers have at least one vortex \cite{Ignat2006260,Millot_energyexpansion}, which looks locally like the radially 
symmetric degree $\pm 1$ solution to the Ginzburg-Landau equation 
\[
\Delta u +u(1-|u|^2)=0, \qquad \mbox{in}\ \R^2.
\]
At still higher values of $\Omega=\mathcal O(\frac{1}{\ve})$ the so called giant vortex 
becomes the equilibrium state of the Bose-Einstein condensate \cite{aftalion5} 
(see also \cite{MR2186426}). All these {localized} structures 
have exact analogues for our one dimensional model. This could be surprising at 
first so let us explain this point. Due to {the} mass  constraint we can recast the Gross-Pitaevskii energy  in the form somewhat similar to $G$
\begin{equation}
F(u)=\int_{\R^2} \frac{1}{2}|\nabla u|^2 +\frac{1}{4\ve^2}\left[\left(|u|^2-a(x)\right)^2-\left(a^-(x)\right)^2\right]^2- \Omega x^\perp\cdot  (iu, \nabla u), 
\label{gp 1}
\end{equation}
where $a(x)=a_0-V(x)$, $a_0$ is determined so that $\int_{\R^2} a^+=1$ and $a^\pm$ 
are the positive and negative parts of the function  $a$. 
Additionally, the  splitting of this functional corresponding to density and phase of $u$  
found in  \cite{LM} shows that on the nonlinear level the two models should have {many properties} in common. To get an idea 
of what we have in mind let us demonstrate the similarity between the case  
when $a=0$ in $E$ and $\Omega=0$ in $F$. The former problem becomes to minimize
\[
E(u)=\int_{\R}\frac{\epsilon}{2}|u_x|^2-\frac{1}{2\epsilon}\mu(x)u^2+\frac{1}{4\epsilon}|u|^4
\]
and the latter to minimize
\[
F(u)=\int_{\R^2} \frac{1}{2}|\nabla u|^2 +\frac{1}{4\ve^2}\left[\left(|u|^2-a(x)\right)^2-\left(a^-(x)\right)^2\right]^2.
\]
Intuitively the global minimizers  should be respectively:  $u=\sqrt{\mu^+}$ and $u=\sqrt{a^+}$ 
(this is the Thomas-Fermi limit of Bose-Einstein condensate).  The problem is that both 
of this functions are not smooth at their zero level sets. Because of this  
the true minimizers will exhibit  a boundary  layer
{behavior}  near the zero level set of $a^+$ or $\mu$ and their local profiles, after suitable scaling,  are given by the unique, positive solution of the 
second Painlev\'e equation \cite{MR555581}
\begin{equation}
\label{PII0}
y'' -xy-2y^3=0, \qquad \mbox{in}\ \R,\\
\end{equation}
such that 
\begin{equation}
y(x)\to 0, \quad x\to \infty, \quad y(x)\sim \sqrt{-x/2}\quad x\to -\infty.
\label{PII asymp}
\end{equation}
This phenomenon is also known as the corner layer and it is present in the context of the Bose-Einstein condensates \cite{MR2062641,MR3355003} as well as in   many other problems,  see for example \cite{alikoakos_1,sourdis1,sourdis0,sourdis2,sourdis3}.  In the next section we will see that the shadow kink, which is the one dimensional analog of the shadow vortex and is the global minimizer of $E(u)$ is described  locally by a solution of the second Painlev\'e equation
\begin{equation}
\label{PII}
y''-xy-2y^3-\alpha=0, \qquad \mbox{in}\ \R,
\end{equation}
with $\alpha\neq 0$ leading to a quite different behaviour than the corner layer. Equation (\ref{PII}) has been  studied by Painlev\'e and others since the early 1900's and is a part of a hierarchy of the Painlev\'e equations, which in turn is a part of a larger hierarchy of equations characterised by the fact that the only movable singularities of their solutions are poles (see for example the monograph \cite{book:1414949}). One of the most interesting aspects of these equations is how ubiquitous they are in applications. To  mention a few examples besides the Bose-Einstein condensates discussed above: the problem of finding self-similar solution of the KdV equation is reduced to (\ref{PII}) by a change of variables (see \cite{MR1149378} and \cite{Flaschka1980} for  more about the connection of (\ref{PII}) with the theory of integrable systems); the theory of random matrices \cite{2006math.ph...3038D}; superconductivity \cite{doi:10.1137/S0036139993254760} \cite{helffer1998}, \cite{PALAMIDES2003177}; for even more  applications we refer to \cite{0951-7715-16-1-321}, \cite{KUDRYASHOV1997397},   \cite{Senthilkumaran20103412} and the references therein. 

In  view of this discussion  existence of the shadow kink should have consequences that go beyond the one dimensional model (\ref{funct 00}) considered here. Indeed our result suggests that   (\ref{PII}) with $\alpha\neq 0$ should play an important role  in various boundary layer phenomena and for this it is necessary to understand special solutions of the Painlev\'e equation beyond the case $\alpha=0$. In fact one of our contributions in this paper is to find new solutions  of  (\ref{PII})   as we explain below. Furthermore, the analogy between the problem of  minimization of the energy functionals  $E$ and  $G$, on the one hand, and formal relation between  $E$ and the Gross-Pitaevski energy functional, on the other hand, suggest that the behaviour of the Bose-Einstein condensates between the threshold values of the angular velocity $\Omega_1<\Omega_2$ is described by a new type of  topological defect, the shadow vortex.  Therefore   it  is important to show rigorously existence of shadow vortices for the energy $G$ and here we make the first step in this direction considering a simpler case of the  energy $E$. 

To explain this let us  briefly discuss  one of the  results of this paper which deals directly with the second Painlev\'e equation (\ref{PII}) and shows existence of a new type of solution. In \cite{MR555581} Hastings and McLeod considered (\ref{PII0}) and showed existence of  a unique solution with (\ref{PII asymp}) as the asymptotic conditions at $\pm \infty$. Here we give another proof of the existence part of this  result based on the fact that when $a=0$ in (\ref{funct 00}) we can identify the local profile of the global minimizer of $E$ in the singular limit $\epsilon \to 0$. In fact our method allows as well to treat  equation (\ref{PII}) and to obtain existence of a generalized solution of Hastings-McLeod (see Theorem \ref{th3n} below). To our knowledge this  result, which was conjectured on the basis of numerical simulations in \cite{MR2367412}, is new. This new solution of the second Painlev\'e equation gives formally the local profile of the shadow vortex which is different from the corner layer type of behaviour determined by (\ref{PII0}). We conjecture that minimizers of the Gross-Pitaevski energy in the intermediate regime $\Omega_1<\Omega<\Omega_2$ may also have similar profile near the zero level set of the function $a(x)$ in (\ref{gp 1}). 

\subsection{Statements of the main results}
\label{main-results}

\

More generally than in (\ref{phys relev}) in what follows we assume that:
\begin{align}
\label{hyp3}
\left\{
\begin{aligned}
&\text{$\mu \in C^1(\R) \cap L^\infty(\R)$ is even, $\mu'<0$ in $(0,\infty)$, and $\mu(\xi)=0$ for a unique $\xi>0$,
}\\
& \text{$f \in L^1(\R)\cap L^\infty(\R)\cap C(\R)$ is odd, $f(x)>0$, $\forall x>0$.}
\end{aligned}
\right.
\end{align}
The assumption  that $\mu$ is even is made here  for the sake of simplicity.  Our statements can easily be adjusted if $\mu'<0$ in $(0,\infty)$, $\mu'>0$ in $(-\infty,0)$, $\mu(\xi)=0$ for a unique $\xi>0$, $\mu(\xi')=0$ for a unique $\xi'<0$.

We consider the energy
\begin{equation}\label{funct 0}
E(u)=\int_{\R}\left(\frac{\epsilon}{2}|u'(x)|^2-\frac{1}{2\epsilon}\mu(x)u^2(x)+\frac{1}{4\epsilon}|u(x)|^4-a f(x)u(x)\right)\dd x, \ u \in H^1(\R).
\end{equation}
In this paper we will keep $a\geq 0$ fixed and $\ve\ll 1$. 
Under assumptions \eqref{hyp3}, there exists $v \in H^1(\R)$ such that $E(v)=\min_{H^1(\R)} E$. In addition, $v\in C^2(\R)$ is a classical solution of the O.D.E.
\begin{equation}\label{ode}
\epsilon^2 v''(x)+\mu(x) v(x)-v^3(x)+\epsilon a f(x)=0, \qquad \forall x\in \R.
\end{equation}
Note that due to the symmetries in \eqref{hyp3}, the energy \eqref{funct 0} and equation 
\eqref{ode} are invariant under the odd symmetry $v(x)\mapsto -v(-x)$.

Next we discuss the dependence of the global minimizer on $a$.
\begin{theorem}\label{th1n}
The following statements hold.
\begin{itemize}
\item[(i)]
When $a=0$ the global minimizer $v$ is even, and positive up to change of $v$ by $-v$.
\item[(ii)] For $a>0$, the global minimizer $v$ has a unique zero $\bar x$ such that
\begin{equation}\label{soph}
|\bar x|\leq \xi+\mathcal{O}(\sqrt{\ve}), \text{ and $v(x)>0$, $\forall x >\bar x$, while $v(x)<0$, $\forall x<\bar x$.}
\end{equation}
\item[(iii)]
Suppose that 
\begin{equation}
a^*:=\sup_{x\in[-\xi,0)} \frac{\sqrt{2}\big((\mu(0))^{3/2}-(\mu(x))^{3/2}\big)}{3\int_x^0|f|\sqrt{\mu}}<\infty. 
\label{astar}
\end{equation}
For all $a>a^*$, $\bar x \to 0$ as $\ve\to 0$, and the global minimizer $v$ satisfies
\begin{equation}\label{cvn1}
\begin{aligned}
&\lim_{\ve\to 0} v(\bar x+\ve s)=\sqrt{\mu(0)}\tanh(s\sqrt{\mu(0)/2}), \medskip\\
&\lim_{\ve\to 0} v( x+\ve s)=\begin{cases}\sqrt{\mu(x)}  &\text{for } 0<x<\xi, \\
-\sqrt{\mu(x)}  &\text{for } -\xi<x<0, \\
0 &\text{for } |x| \geq \xi,
\end{cases}
\end{aligned}
\end{equation}
in  the $C^1_{\mathrm{loc}}(\R)$ sense. 
\item
[(iv)]
Let 
$$
a_*:=\inf_{x \in (-\xi,0]}\frac{\sqrt{2}(\mu(x))^{3/2}}{3\int_{-\xi}^x |f|\sqrt{\mu} }\in (0,\infty).
$$
Up to change of $v(x)$ by $-v(-x)$, for all $a\in (0, a_*)$, $\bar x\to -\xi$ as $\ve\to 0$,  and
\begin{equation}\label{cvn2}
\lim_{\epsilon\to 0} v(x+s\epsilon)=
\begin{cases}
\sqrt{\mu(x)}  &\text{for } |x|<\xi, \\
0 &\text{for } |x| \geq \xi,
\end{cases}
\end{equation} 
in  the $C^1_{\mathrm{ loc}}(\R)$ sense.
The above asymptotic formula holds as well when $a=0$. Moreover, when $f=-\frac{\mu'}{2}$ we have $a_*=a^*=\sqrt{2}$.  
\end{itemize}
\end{theorem}
We observe that (\ref{astar})  holds for instance provided $\mu$ is twice differentiable at $0$, and $f'(0)>0$ cf. Step \ref{s6} of the proof below.

The preceding theorem justifies the name {\it shadow kink} for the global minimizer when $a\in (0, a_*)$.  Indeed, when $a>a^*$ the global minimizer has a profile of suitably re-scaled and modulated hyperbolic tangent. This is not surprising since $H(x)=\tanh(x/\sqrt{2})$ is a solution of the Allen-Cahn equation
\begin{equation}\label{ac}
H''+H-H^3=0, \qquad \mbox{in}\ \R,
\end{equation}
and it is a standard, local profile of topological defects such as kinks or domain walls appearing in many phase transition problems. On the other hand, when $a<a_*$ the zero of the global minimizer occurs near the  point where $\xi$ changes its sign i.e. between the illuminated zone and the dark zone in the nematic liquid crystal experiment. Because of this,  unlike in the case of the standard kink, the shadow kink is hard to  detect experimentally.

Next we will study local profiles of the global minimizers near the points $\pm \xi$, that is the zeros of $\mu$. Our goal is to show that the shadow kink is indeed different than the standard kink, and  its local profile near the point of sign change is nothing like the solution (\ref{ac}). We recall the second Painlev\'e equation
\begin{equation}\label{pain}
y''(s)-s y(s)-2y^3(s)-\alpha=0, \qquad \forall s\in \R.
\end{equation}
We will now define the notion of  minimal solutions of (\ref{pain}). Let us denote
\[
E_{\mathrm{P_{II}}}(u, I)=\int_I \left[ \frac{1}{2} |u'|^2 +\frac{1}{2}  s u^2 +\frac{1}{2} u^4+\alpha u\right]
\]
By definition  a  solution of (\ref{pain}) is minimal if
\[
E_{\mathrm{P_{II}}}(y, \mathrm{supp}\, \phi)\leq E_{\mathrm{P_{II}}}(y+\phi, \mathrm{supp}\, \phi)
\]
for all $\phi\in C^\infty_0(\R)$. This notion of minimality is standard for many problems in which the energy of a localized solution is actually infinite due to non compactness of the domain. 
\begin{theorem}[Local profile of the global minimizer]\label{th2n}
Let $v$ be the global minimizer of $E$ for $a\geq 0$, let $\mu_1:=\mu'(\xi)<0$, and let
\[
w^\pm(s)=\pm 2^{-1/2}(-\mu_1\ve)^{-1/3} v\Big(\pm \xi\pm\ve^{2/3} \frac{s}{(-\mu_1)^{1/3}}\Big).
\]
As $\ve\to 0$, the function $w^\pm$ converges in $C^1_{\mathrm{loc}}(\R)$ up to subsequence, to a bounded at $\infty$, minimal solution of \eqref{pain} with $\alpha=\frac{a f(\xi)}{\sqrt{2}\mu_1}<0$. 
\end{theorem}
In order to be more precise about the limit of $w^\pm$ we state:
\begin{theorem}[A generalisation of the Hastings-McLeod result]\label{th3n}
The following statements hold.
\begin{itemize}
\item[(i)]
For any $\alpha \leq 0$\footnote{By changing $y$ by $-y$, we obtain the solutions of \eqref{pain} corresponding to $\alpha\geq 0$} the second Painlev\'e equation has a positive minimal solution $y$, which is strictly decreasing ($y' <0$) and such that 
\begin{itemize}
\item[(a)]
When $\alpha=0$
\begin{align}\label{asy0}
y(s)&\sim \mathop{Ai}(s), \qquad s\to \infty \nonumber \\
y(s)&\sim \sqrt{|s|/2}, \qquad s\to -\infty
\end{align}

Moreover, this is the only nonnegative minimal solution, bounded at $\infty$.
\item[(b)] When $\alpha<0$ 
\begin{align}\label{asy1}
y(s)&\sim \frac{|\alpha|}{s}, \qquad s\to \infty \nonumber\\
y(s)&\sim \sqrt{|s|/2}, \qquad s\to -\infty
\end{align}

\end{itemize}
\item[(ii)]
When $\alpha<0$ and $y$ is a minimal solution bounded at $\infty$, such that it vanishes at $s=\bar s$ then
\begin{align}\label{asy2}
y(s)&\sim \frac{|a|}{s}, \qquad s\to \infty \nonumber\\
y(s)&\sim -\sqrt{|s|/2}, \qquad s\to -\infty
\end{align}
\end{itemize}
\end{theorem}
From this we have as a corollary:
\begin{corollary}\label{cor teo 3}
If $v$ is the global minimizer of $E$ for $a>0$ and  if $v\geq 0$ on $[0,\infty)$ (resp. $v\leq 0$ on $(-\infty,0]$), then $w^+$ (resp. $w^-$) converges to the solution $y$, described in Theorem \ref{th3n} (i).
\end{corollary}

Part (i) (a) of Theorem \ref{th3n} generalizes the result of Hastings-McLeod in the sense that we characterise their solution as minimal. This property holds also for solutions described in part (i) (b) of this theorem and it explains why   they are energetically privileged in the boundary layer behaviour seen in various physical systems. This should be compared with the well known minimality property of $H(x)=\tanh(x/\sqrt{2})$ for the Allen-Cahn equation.   
Existence of the minimal solution described in Theorem  \ref{th3n} (ii) above is conjectured on the basis of  numerical simulations of the global minimizers of $E$. A rigorous proof is an open problem.

In the rest of this paper we give proofs of the results stated above. 
\section{Proof of Theorem \ref{th1n}}

\begin{step}\label{l1n}(Existence of a global minimizer)
\end{step}
\begin{lemma}\label{lem exist min}
There exists $v \in H^1(\R)$ such that $E(v)=\min_{H^1(\R)} E$. As a consequence, $v$ is a classical solution of \eqref{ode}.
\end{lemma}
\begin{proof}
We first show that $\inf\{\,E(u):\ u \in H^1_{\mathrm{loc}}(\R) \, \}>-\infty$. 
To see this, we regroup the last three terms in the integral of $E(u)$. 
Setting $I_\eta:=\{ x \in \R: \mu(x)+\eta>0\}$, for $\eta>0$ sufficiently small such that $I_\eta$ is bounded, we have
$$ -\frac{1}{2\epsilon}\mu(x)u^2+\frac{1}{8\epsilon}|u|^4< 0 \Longleftrightarrow u^2< 4\mu \Longrightarrow x\in I_{\eta},$$
thus
$$ -\frac{1}{2\epsilon}\mu(x)u^2+\frac{1}{8\epsilon}|u|^4\geq -\frac{2}{\epsilon}\left\| \mu\right\|^2_{L^\infty} \chi_\eta,$$
where $\chi_\eta$ is the characteristic function of $I_\eta$. On the other hand,
$$\frac{1}{8\epsilon}|u|^4-a f(x)u<0\Longrightarrow |u|^3\leq 8 a\epsilon|f|\Longrightarrow|fu|\leq  (8a\epsilon)^{1/3}|f|^{4/3},$$
thus
$$\frac{1}{8\epsilon}|u|^4-a f(x)u\geq-a (8a\epsilon)^{1/3}|f|^{4/3}.$$
Next, we notice that $E(u)\in \R$ for every $u \in H^1(\R)$, thanks to the imbedding $H^1(\R)\subset L^p(\R)$, for $2\leq p \leq\infty$. Now, let $m:=\inf_{H^1}E>-\infty$, and let $u_n$ be a sequence such that $E(u_n) \to m$. Repeating the previous computation, we can bound
\begin{align}
\int_\R\frac{\epsilon}{2}| u'_n|^2+\frac{\eta}{2\epsilon}u_n^2 &=E(u_n)+\int_{\R}\frac{1}{2\epsilon}(\mu(x)+\eta)u^2_n-\frac{1}{4\epsilon}|u_n|^4+a f(x)u_n\nonumber \\
&\leq E(u_n)+\frac{2}{\epsilon}(\left\| \mu\right\|_{L^\infty}+\eta)^2 |I_\eta|+a (8a\epsilon)^{1/3}\int_{\R}|f|^{4/3}.\nonumber
\end{align}
From this expression it follows that $\left\|u_n \right\|_{H^{1}(\R)}$ is bounded. As a consequence, for a subsequence still called $u_n$, $u_n\rightharpoonup v$ weakly in $H^1$, and thanks to a diagonal argument we also have $u_n\to v$ in $L^2_{\mathrm{loc}}$, and almost everywhere in $\R$. Finally, by lower semicontinuity $$\int_\R |v'|^2\leq \liminf_{n \to \infty} \int_\R  |u'_n|^2,$$ and by Fatou's Lemma we have $$\int_\R |v|^4\leq \liminf_{n \to \infty} \int_\R  |u_n|^4, \text{ and }\int_{\mu\leq 0}-\frac{1}{2\epsilon}\mu v^2 \leq \liminf_{n \to \infty} \int_{\mu\leq 0}-\frac{1}{2\epsilon}\mu u_n^2.$$ To conclude, it is clear that $$\int_{\mu> 0}-\frac{1}{2\epsilon}\mu v^2= \lim_{n \to \infty} \int_{\mu > 0}-\frac{1}{2\epsilon}\mu u_n^2,$$ thus $m\leq E(v)\leq\liminf_{n \to \infty}  E(u_n)=m$.
\end{proof}
\begin{step}(Proof of (i))
\end{step}\begin{proof}
When $a=0$, we have $E(|v|)=E(v)$, in particular $|v|$ is also a minimizer and a smooth solution of the Euler-Lagrange equation. Now suppose that $v(x_0)=0$ for some $x_0$. Then, $|v|$ has a minimum at $x_0$, and $|v|(x_0)=|v|'(x_0)=0$. By the uniqueness result for O.D.E., it follows that $v\equiv 0$. However, this situation does not occur for $\epsilon \ll 1$. Indeed, by choosing an appropriate test function $\phi$ with $\supp\phi\subset (-\xi,\xi)$, one can see that $E(\phi)<0$ for $\epsilon\ll 1$. Thus $v$ is positive up to change of $v$ by $-v$. Finally, we notice that $E(v,[0,\infty))=E(v,(-\infty,0])$, since otherwise we can construct a function in $H^1$ with smaller energy than $v$. As a consequence, $\tilde v(x) = v(|x|)$ is also a minimizer, and since $\tilde v= v$ on $[0,\infty)$, it follows by the uniqueness result for O.D.E. that $\tilde v\equiv v$.
\end{proof}

\begin{step}(Uniform bounds)
\begin{lemma}\label{s3}
For $\epsilon$ and $a$ belonging to a bounded interval, let $u_{\epsilon,a}$ be a solution of \eqref{ode} converging to $0$ at $\pm\infty$. Then, the solutions $u_{\epsilon,a}$ are uniformly bounded.
\end{lemma}
\end{step}

\begin{proof}
Since $|f|,\mu,\epsilon,a$ are bounded, the roots of the cubic equation $$-u^3+\mu(x)u+\epsilon a f(x)=0$$ belong to a bounded interval, for all values of $x,\epsilon,a$. If $u:=u_{\epsilon,a}$ takes positive values, then it attains its maximum $0\leq \max_{\R}u=u(x_0)$, at a point $x_0\in\R$. Since $$0\geq\epsilon^2 u''(x_0)=u^3(x_0)-\mu(x_0)u(x_0)-\epsilon a f(x_0),$$ we can see that $u(x_0)$ is uniformly bounded above. In the same way, we prove the uniform lower bound.
\end{proof}

\begin{step}(Proof of (ii))
\end{step}
\underline{Claim 1}: When $a>0$, the global minimizer $v$ has at most one zero, denoted by $\bar x$. Furthermore, $v(x)>0$, $\forall x >\bar x$, and $v(x)<0$, $\forall x<\bar x$. 
\begin{proof}[Proof of Claim 1]
Let $\bar x\geq 0$ be a zero of $v$. If $v(x_1)<0$ for some $x_1>\bar x$, then $E(v ,[\bar x,\infty)>E(|v|,[\bar x,\infty))$, which is a contradiction. Now, if $v(x_2)=0$ for some $x_2>\bar x$, then according to what precedes $v$ has a minimum at $x_2$. It follows that $v(x_2)=v'(x_2)=0$, and $v''(x_2) \geq 0$, which is impossible, since by \eqref{ode} we have: $\epsilon v''(x_2)=-af(x_2)<0$. Thus we have proved that $v(\bar x)=0$, with $\bar x \geq 0$, implies that $v(x)>0$, $\forall x>\bar x$. Thanks to the previous argument, we also see that $v$ cannot have another zero in the interval $[0,\infty)$. In the same way, one can show that $v$ has at most one zero $\bar y$ in the interval $(-\infty,0]$. Furthermore, $v(\bar y)=0$, with $\bar y \leq 0$, implies that $v(x)<0$, $\forall x<\bar y$. To complete the proof, it remains to exclude the case where $v(\bar y)=v(\bar x)=0$, with $\bar y<0<\bar x$. In this case, we have either $v>0$ or $v<0$ in the interval $(\bar y, \bar x)$. 
Assuming the former we see that $v$ has a minimum at $\bar x$, which is impossible by the argument at the beginning of the proof.
The second statement of Claim 1 follows by a similar argument. 
\end{proof}
\underline{Claim 2}: If $\epsilon>0$ and $a>0$ remain in a bounded interval, there exists a constant $\delta>0$ such that $v$ has a unique zero $\bar x$, when $\frac{\epsilon}{a}<\delta$. In addition, $|\bar x|\leq \xi+\mathcal O(\sqrt{\epsilon/a})$. 
\begin{proof}[Proof of Claim 2]
Suppose that $x_0<-\xi$ is such that $v(x_0)> 0$. We are first going to show that $v'(x_0)>0$. Indeed, suppose by contradiction that $v'(x_0)\leq 0$. Setting $$m:=\inf\{ x<x_0: \ v \geq 0\text{ on }[x,x_0]\},$$
one can see by \eqref{ode}, that $v$ is convex on the interval $(m,x_0]$, and thus $v \geq  v(x_0)$ on $(m,x_0]$. It follows that $m=-\infty$, which is a contradiction since $\lim_{-\infty}v=0$. This proves our claim. Now, let $M>0$ be the constant (cf. Lemma \ref{s3}), such that $|v_{\epsilon,a}|\leq M$ when $\epsilon$ and $a$ remain bounded, and let $m'=\min_{[-\xi-1,-\xi]}(-f)$. According to Claim 1, we have $v>0$ on the interval $[x_0,-\xi]$, thus in view of \eqref{ode} we have $v''\geq -\frac{a}{\epsilon}f$ on $[x_0,-\xi]$. In particular, for any  $x_0 \in [-\xi-1,-\xi]$ such that $v(x_0)> 0$, we obtain 
\begin{equation}\label{equzero1}
M\geq v(-\xi)-v(x_0)\geq m' \frac{a}{\epsilon} \frac{(\xi+x_0)^2}{2}.
\end{equation}
From this inequality, we see by taking $\delta=\frac{m'}{2M}$, that if $\frac{\epsilon}{a}<\delta$, then we cannot have $x_0=-\xi-1$, or in other words  $v(-\xi-1)\leq 0$. Repeating the same analysis for $x_0>\xi$, we also deduce that if  $\frac{\epsilon}{a}<\delta$ then  $v(\xi+1)\geq 0$. Thus, the existence of a zero of $v$ in the interval $[-\xi-1,\xi+1]$ is ensured when $\frac{\epsilon}{a}<\delta$. This zero denoted by $\bar x$ which is unique by Claim 1, satisfies in view of \eqref{equzero1}: $|\bar x|\leq \xi+\mathcal O(\sqrt{\epsilon/a})$. 
\end{proof}

\begin{step}(Upper bound of the renormalized energy)

\end{step}

The minimum of the energy defined in \eqref{funct 0} is nonpositive and tends to $-\infty$ as $\epsilon \to 0$. Since we are interested in the behavior of the minimizers as $\epsilon \to 0$, it is useful to define a renormalized energy, which is obtained by adding to \eqref{funct 0} a suitable term so that the result is bounded from below and above by an $\ve$ independent constant. We define the renormalized energy as 
\begin{equation}\label{renorm}
\mathcal{E}(u):=E(u)+\int_{|x|<\xi}\frac{\mu^2}{4\epsilon}=\int_{\R}\frac{\epsilon}{2}|u'|^2+\int_{|x|<\xi}
\frac{(u^2-\mu)^2}{4\epsilon}+\int_{|x|>\xi}\frac{u^2(u^2-2\mu)}{4\epsilon}-\int_\R a f u ,
\end{equation}  
and claim  the bound 
\begin{equation}\label{qua}
\limsup_{\epsilon\to 0}\mathcal{E}(v_{\epsilon,a})\leq \min\Big(0,\frac{2\sqrt{2}}{3}(\mu(0))^{3/2}-\int_{-\xi}^\xi a |f|\sqrt{\mu}\Big).
\end{equation}
\begin{proof}[Proof of \eqref{qua}]
Let us consider the $C^1$ piecewise function:
\begin{equation}
\phi(x)=
\begin{cases}
\sqrt{\mu(x)} &\text{for } |x|\leq \xi-\epsilon, \\
k_\epsilon \epsilon^{1/2}e^{-\frac{|x|-\xi}{\epsilon}}  &\text{for } |x|\geq \xi-\epsilon,
\end{cases}\nonumber
\end{equation} 
with $k_\ve$ defined by $k_\epsilon \epsilon^{1/2}e=\sqrt{\mu(\xi- \epsilon)}\Longrightarrow k_\epsilon=\mathcal O(1)$.
Since $\phi \in H^1(\R)$, it is clear that $\mathcal E(v)\leq \mathcal E(\phi)$. We check that $\mathcal E(\phi)=\mathcal O(\epsilon \ln(\epsilon))$, since it is the sum of the following integrals:
$$\int_{|x|\leq\xi-\epsilon}\frac{\epsilon}{2}\frac{ |\mu'|^2}{4\mu}=\mathcal O(\epsilon | \ln \epsilon |), \ \int_{\xi-\epsilon\leq|x|\leq \xi}\frac{\mu^2}{4\epsilon}=\mathcal O(\epsilon^2), $$
$$E(\phi, (-\infty,-\xi+\epsilon])+ E(\phi, [\xi-\epsilon,\infty))=\mathcal O(\epsilon).$$ 
Next, we repeat the previous computation by considering another $C^1$ piecewise function:
\begin{equation}
\psi(x)=
\begin{cases}
-k_\epsilon \epsilon^{1/2}e^{\frac{x+\xi}{\epsilon}}  &\text{for } x\leq -\xi+\epsilon,\\
-\sqrt{\mu(x)} &\text{for } -\xi+\epsilon\leq x\leq -\zeta_\epsilon \epsilon, \\
l_\epsilon \tanh\Big(\frac{x}{\epsilon}\sqrt{\frac{\mu(0)}{2}}\Big) &\text{for } |x|\leq \zeta_\epsilon \epsilon, \\
\sqrt{\mu(x)} &\text{for } \zeta_\epsilon \epsilon \leq x\leq \xi-\epsilon, \\
k_\epsilon \epsilon^{1/2}e^{-\frac{x-\xi}{\epsilon}}  &\text{for } x\geq \xi-\epsilon,
\end{cases}\nonumber
\end{equation} 
with $$\zeta_\epsilon=- \ln \epsilon,\ k_\epsilon \text{ as above},$$ 
$$l_\epsilon \tanh\Big(\zeta_\epsilon\sqrt{\frac{\mu(0)}{2}}\Big)=\sqrt{\mu(\zeta_\epsilon\epsilon)}\Longrightarrow \lim_{\epsilon \to 0}l_\epsilon=\sqrt{\mu(0)}, \ \frac{l_\epsilon^2}{\mu(\zeta)}=1+O(\epsilon^\gamma), \text{ for some $0<\gamma<1$}.$$ 
Since $\psi \in H^1(\R)$, we have $\mathcal E(v)\leq \mathcal E(\psi)$. We can check that 
\begin{equation}\label{quabb}
\lim_{\epsilon \to 0}\mathcal E(\psi)\to\frac{2\sqrt{2}}{3}(\mu(0))^{3/2}-\int_{-\xi}^\xi a |f|\sqrt{\mu}.
\end{equation}
Indeed, setting $\tilde \psi(s)=\sqrt{\mu(0)}\tanh\Big(s\sqrt{\frac{\mu(0)}{2}}\Big)$, $\mathcal E(\psi)$ is the sum of the following integrals:
$$E(\psi, (-\infty,-\xi+\epsilon])+ E(\psi, [\xi-\epsilon,\infty))=\mathcal O(\epsilon),$$ 
$$\int_{\zeta_\epsilon\epsilon<|x|\leq\xi-\epsilon}\frac{\epsilon}{2}\frac{ |\mu'|^2}{4\mu}=\mathcal O(\epsilon | \ln \epsilon |), \ \int_{\xi-\epsilon\leq|x|\leq \xi}\frac{\mu^2}{4\epsilon}=\mathcal O(\epsilon^2), $$
$$-\int_{\zeta_\epsilon \epsilon<|x|<\xi-\epsilon}a |f|\sqrt{\mu}\to-\int_{|x|<\xi}a |f|\sqrt{\mu},$$
$$\int_{|x|\leq \zeta_\epsilon \epsilon}\frac{\mu^2}{4\epsilon}=\int_{|x|\leq \zeta_\epsilon \epsilon}\frac{\mu^2(0)}{4\epsilon}+\mathcal O(\zeta_\epsilon^2\epsilon), $$
$$\int_{|x|<\zeta_\epsilon \epsilon}\frac{\epsilon}{2}|\psi'|^2=\frac{l_\epsilon^2}{\mu(\zeta)}\int_{|s|<\zeta_\epsilon}\frac{1}{2}|\tilde \psi'|^2 =\int_{|s|<\zeta_\epsilon}\frac{1}{2}|\tilde \psi'|^2+\mathcal O(\zeta_\epsilon \epsilon^\gamma),$$   
$$-\int_{|x|<\zeta_\epsilon \epsilon}\frac{\mu}{2\epsilon}\psi^2=-\int_{|x|<\zeta_\epsilon \epsilon}\frac{\mu(0)}{2\epsilon}\psi^2+\mathcal O(\zeta_\epsilon^2\epsilon)=-\frac{l_\epsilon^2}{\mu(0)}\int_{|s|<\zeta_\epsilon }\frac{\mu(0)}{2}\tilde \psi^2+
\mathcal O(\zeta_\epsilon^2\epsilon)=-\int_{|s|<\zeta_\epsilon }\frac{\mu(0)}{2}\tilde \psi^2+\mathcal O(\zeta_\epsilon \epsilon^\gamma),$$
$$\int_{|x|<\zeta_\epsilon \epsilon}\frac{1}{4\epsilon}|\psi|^4=\frac{l_\epsilon^4}{(\mu(\zeta))^2}\int_{|s|<\zeta_\epsilon }\frac{1}{4}|\tilde \psi|^4=\int_{|s|<\zeta_\epsilon }\frac{1}{4}|\tilde \psi|^4+\mathcal O(\zeta_\epsilon \epsilon^\gamma),$$
$$-\int_{|x|<\zeta_\epsilon \epsilon}a f \psi=\mathcal O(\zeta_\epsilon \epsilon).$$
Gathering the previous equations, \eqref{quabb} follows immediately.
\end{proof}

\begin{step}\label{s6}
Let $a>0$, and let $v_{\epsilon,a}$ be a global minimizer. Up to the odd symmetry we may assume that $v$ is nonnegative on $[0,\infty)$. Setting
$$a_*:=\inf_{x \in (-\xi,0]}\frac{\sqrt{2}(\mu(x))^{3/2}}{3\int_{-\xi}^x |f|\sqrt{\mu} }\in (0,\infty),$$ 
and
$$a^*:=\begin{cases}\sup_{x\in[-\xi,0)} \frac{\sqrt{2}\big((\mu(0))^{3/2}-(\mu(x))^{3/2}\big)}{3\int_x^0|f|\sqrt{\mu}},\quad \mbox{if}\ \sup\ \mbox{exists},\\
+\infty, \quad \mbox{otherwise},
\end{cases}$$
we have $\bar x \to -\xi$ as $\epsilon \to 0$, and $a\in(0,a_*)$, while $\bar x \to 0$ as $\epsilon \to 0$, and $a>a^*$.  In the particular case where $f=-\frac{\mu'}{2}$, we have $a_*=a^* =\sqrt{2}$.
\end{step}
\begin{proof}
Let us consider a sequence $\epsilon_n \to 0$, let $a>0$, and suppose that $\bar x_n:=\bar x_{\epsilon_n,a}\to l \in [-\xi,\xi]$, as $n\to \infty$ (cf. \eqref{soph}).
We rescale $v$ by setting $\tilde v_n(s)=v_{\epsilon_n,a}(\bar x_{n}+s\epsilon_n)$. Clearly, $\tilde v''_n(s)=\epsilon_n^2 v''_{\epsilon_n,a}(\bar x_{n}+s\epsilon_n)$. As a consequence of Lemma \ref{s3} and \eqref{ode}, the functions $\tilde v_n$ are uniformly bounded up to the second derivatives. Thus, we can apply the theorem of Ascoli, via a diagonal argument, and show that for a subsequence still called $\tilde v_n$,
$\tilde v_{n}$ converges in $C^1_{\mathrm{ loc}}(\R)$ to a function $\tilde V$. Now, we are going to determine $\tilde V$.
For this purpose, we introduce the rescaled energy
\begin{equation}
\label{functres}
\tilde E_n(\tilde u)=\int_{\R}\Big(\frac{1}{2}|\tilde u'(s)|^2-\frac{1}{2}\mu(\bar x_n+s\epsilon_n)\tilde u^2(s)+\frac{1}{4}|\tilde u|^4(s)-\epsilon_n a f(\bar x_n+s\epsilon_n)\tilde u(s)\Big)\dd s=E(u_n),\nonumber
\end{equation}
where we have set $\tilde u(s)=u_{n}(\bar x_{n}+s\epsilon_n)$ i.e. $u_n(x)=\tilde u\big(\frac{x-\bar x_n}{\epsilon_n}\big)$.
Let $\tilde \xi$ be a test function with support in the compact interval $J$. We have $\tilde E_{n}(\tilde v_{n}+\tilde \xi,J)\geq \tilde E_{n}(\tilde v_{n},J)$, and at the limit
$G_{0}( \tilde V+\tilde\xi,J)\geq  G_{0}(\tilde V,J)$, where $$ G_{0}(\phi,J)=\int_{J}\left[\frac{1}{2}|\phi'|^2-\frac{1}{2}\mu(l)\phi^2+\frac{1}{4}|\phi |^4\right],$$
or equivalently $G( \tilde V+\tilde\xi,J)\geq G(\tilde V,J)$, where
\begin{equation}\label{gl}
G(\phi,J)=\int_{J}\left[\frac{1}{2}|\phi'|^2-\frac{1}{2}\mu(l)\phi^2+\frac{1}{4}|\phi|^4+\frac{(\mu(l))^2}{4}\right]
=\int_{J}\left[\frac{1}{2}|\phi'|^2+\frac{1}{4}(\phi^2-\mu(l))^2\right].
\end{equation}
Thus, we deduce that $\tilde V$ is a bounded minimal solution of the O.D.E. associated to the functional \eqref{gl}:
\begin{equation}\label{odegl}
\tilde V''(s)-(\tilde V^2(s)-\mu(l))\tilde V(s)=0,
\end{equation}
and since we have $\tilde V(0)=0$, and $\tilde V(s)\geq 0$, $\forall s \geq 0$, we obtain $\tilde V(s)=\sqrt{\mu(l)}\tanh(s\sqrt{\mu(l)/2})$. 
So far we have proved that 
\begin{equation}
\lim_{n\to \infty} v(\bar x_n+\ve_n s)=\sqrt{\mu(l)}\tanh(s\sqrt{\mu(l)/2}), \text{ in the $C^1_{\mathrm{loc}}$ sense}.
\end{equation}
Similarly, one can show that 
\begin{equation}\label{nn3}
\lim_{n\to \infty} v(x+\ve_n s)=\begin{cases}\sqrt{\mu(x)}  &\text{for } l<x<\xi, \\
-\sqrt{\mu(x)}  &\text{for } -\xi<x<l, \\
0 &\text{for } |x| \geq \xi,
\end{cases}\text{ in the $C^1_{\mathrm{loc}}$ sense}.
\end{equation}

Next, we compute a lower bound of the renormalized energy of $v_n$, by examining each integral appearing in the definition of $\mathcal E$ (cf. \eqref{renorm}). 
In view of Lemma \ref{s3} and \eqref{nn3}, we have by dominated convergence
\begin{equation}
\lim_{n\to \infty} -\int_\R a f v_n= \int_{-\xi}^l a f \sqrt{\mu} -\int_l^\xi a f \sqrt{\mu}. \nonumber
\end{equation}  
On the other hand, it is clear that
\begin{equation}
0\leq\int_{|x|>\xi}\frac{v_n^2(v_n^2-2\mu)}{4\epsilon},\nonumber
\end{equation}  
and
\begin{equation}
\int_{\R}\Big(\frac{\epsilon}{2}|v_{n}'|^2+\chi_{(-\xi,\xi)}\frac{(v^2_{n}-\mu)^2}{4\epsilon}\Big) 
=\int_{\R}\Big(\frac{1}{2}|\tilde v_{n}'(s)|^2+\chi_{(-(\xi+\bar x_n)\epsilon^{-1}_n,(\xi-\bar x_n)\epsilon^{-1}_n)}(s)\frac{(\tilde v^2_{n}(s)-\mu( \bar x_n+ s\epsilon_n))^2}{4}\Big)\dd s=:L_n,  \nonumber
\end{equation} 
where $\chi$ is the characteristic function. Finally, by Fatou's Lemma, we obtain
\begin{align}
\liminf_{n\to \infty}L_n &\geq \int_{\R} \liminf_{n\to \infty}
\Big(\frac{1}{2}|\tilde v_{n}'(s)|^2+\chi_{(-(\xi+\bar x_n)\epsilon^{-1}_n,(\xi-\bar x_n)\epsilon^{-1}_n)}(s)\frac{(\tilde v^2_{n}(s)-\mu( \bar x_n+ s\epsilon_n))^2}{4}\Big) =\frac{2\sqrt{2}}{3}(\mu(l))^{3/2}.\nonumber
\end{align}
Thus,
\begin{align*}
\liminf_{n\to \infty} \mathcal{E}(v_{n})\geq\frac{2\sqrt{2}}{3}(\mu(l))^{3/2}+\int_{-\xi}^l a f \sqrt{\mu} -\int_l^\xi a f \sqrt{\mu}.
\end{align*}

To conclude, we are going to compare the above lower bound with the upper bound \eqref{qua}, and deduce the convergence of the zero of the minimizer according to the value of $a$.
We first check that $a_*>0$. Let $\psi: [-\xi,0]\ni x \mapsto \psi (x)= \frac{\sqrt{2}}{3}(\mu(x))^{3/2}-a \int_{-\xi}^x |f|\sqrt{\mu}$. There exists $a_1$ such that for $0<a<a_1$ we have $\psi'> 0$ on a small interval $(-\xi,-\xi+\gamma]$, with $\gamma>0$. Also, there exists $a_2$ such that for $0<a<a_2$, we have $\psi > 0$ on $[-\xi+\gamma,0]$. Thus, we can see that $a_*\geq\min(a_1,a_2)$. Now, if the minimizers $v_n$ are nonnegative on $[0,\infty)$, it follows that $l \in [-\xi,0]$, and that $\liminf_{n\to \infty} \mathcal{E}(v_{n})\geq\frac{2\sqrt{2}}{3}(\mu(l))^{3/2}-2\int_{-\xi}^l a |f| \sqrt{\mu}>0$, for $l \in (-\xi,0]$ and $a \in (0,a_*)$. In view of \eqref{qua} in Step \ref{s6}, this situation does not occur, hence $\bar x_{\epsilon,a} \to -\xi$ as $\epsilon \to 0$, and $a\in(0,a_*)$. Similarly,
$\liminf_{n\to \infty} \mathcal{E}(v_{n})\geq\frac{2\sqrt{2}}{3}(\mu(l))^{3/2}-2\int_{-\xi}^l a |f| \sqrt{\mu}>\frac{2\sqrt{2}}{3}(\mu(0))^{3/2}-2\int_{-\xi}^0 a |f| \sqrt{\mu}$, for $l \in [-\xi,0)$ and $a >a^*$. Again, by  \eqref{qua}, this situation does not occur, hence $\bar x_{\epsilon,a} \to 0$ as $\epsilon \to 0$, and $a>a^*$. When $f=-\frac{\mu'}{2}$, an easy computation shows that $a_*=a^* =\sqrt{2}$.
\end{proof}
\begin{step}\label{s7}(Proof of \eqref{cvn1} and \eqref{cvn2})
\end{step}
\begin{proof}
We proceed as in Step \ref{s6}.
For fixed $a \geq 0$, and $\epsilon_n \to 0$, we consider the sequence of global minimizers $v_n:=v_{\epsilon_n,a}$, and rescale them by setting $\tilde v_n(s)=v(x+\epsilon_n s)$. Since the rescaled sequence $\tilde v_n$ is uniformly bounded up to the second derivatives (cf. Lemma \ref{s3}), we obtain the convergence in $C^1_{\mathrm{loc}}$ of a subsequence to a minimal solution $\tilde V$ of the O.D.E. $\tilde V''=W'(\tilde V)$. According to the shape of the potential $W$, and to the location of the zero of $v$, we deduce that $\tilde V$ is either a constant or a heteroclinic connection (cf. \cite{antonop}). Finally, since the limit  $\tilde V$ is independent of the sequence $\epsilon_n$, we obtain the convergence in \eqref{cvn1} and \eqref{cvn2}, as $\epsilon\to 0$.
\end{proof}

\section{Proof of Theorems \ref{th2n} and \ref{th3n}}

\setcounter{step}{0}

\begin{step}\label{s3b}(Uniform bounds)
\begin{lemma}\label{l2}
For $\epsilon\ll 1$ and $a$ belonging to a bounded interval, let $u_{\epsilon,a}$ be a solution of \eqref{ode} converging to $0$ at $\pm\infty$. Then, there exist a constant  $K>0$ such that 
\begin{equation}\label{boundd}
|u_{\ve,a}(x)|\leq K(\sqrt{\max(\mu (x),0)}+\ve^{1/3}), \quad \forall x\in \R.
\end{equation}
As a consequence, the rescaled functions $\tilde u_{\epsilon,a}^\pm(s)=\pm\frac{u_{\ve,a}(\pm\xi\pm s\epsilon^{2/3})}{\epsilon^{1/3}}$ are uniformly bounded on the intervals $[s_0,\infty)$, $\forall s_0\in\R$. 
\end{lemma}
\end{step}
\begin{proof}
For the sake of simplicity we drop the indexes and write $u:=u_{\epsilon,a}$. Let $ M>0$ be the constant such that $|u_{\epsilon,a}|$ is uniformly bounded by $M$ (cf. Lemma \ref{s3}), and let $k>0$ be such that $4 \mu(\xi+h)<-kh<8\mu(\xi+h)$, for $h \in (\xi -\delta,\xi)$ (with $\delta>0$ small). Next, define $\lambda>1$ such that $\lambda k \delta\geq M^2$. Finally, let $F:=\sup f$.
To prove the uniform upper bound for $x \geq 0$, we utilize the strict convexity of $u$ in the region $$D:=\left\{(x,y)\in [0,\infty)\times [0,\infty): \, y>\sqrt{4\max(\mu(x),0)}+(4\epsilon a F)^{1/3}\right\}.$$ 
Indeed, one can see that for $x \geq 0$, the positive root $\sigma$ of the cubic equation $u^3-\mu(x)u-\epsilon a f(x)=0$, satisfies 
$\sigma (x)-\sqrt{\mu(x)}\leq |\epsilon  a f(x)|^{1/3}$, $\forall x\in [0,\xi]$, and $\sigma(x)\leq |\epsilon af(x)|^{1/3}, \forall x\geq \xi$.

Suppose there exists $x_0 \in [0,\xi-\epsilon^{2/3})$ such that $u(x_0)\geq \sqrt{8\lambda \mu(x_0)}+(4\epsilon a F)^{1/3}$. In view of what precedes we have $x_0 \in (\xi-\delta,\xi)$. In fact, we are going to show that $|\xi-x_0|\leq K'\epsilon^{2/3}$, for a constant $K'>0$.
Our claim is that
\begin{equation}\label{claim1c}
u(z)>\sqrt{ 4\mu(z)}+(4\epsilon a F)^{1/3}, \ \text{ for $x_0\leq z \leq \xi$.}
\end{equation}
Indeed, if $u(x_2)\leq\sqrt{ 4\mu(x_2)}+(4\epsilon a F)^{1/3}$, for some $x_2 \in (x_0,\xi]$, the curve $[0,\xi] \ni x \to \sqrt{k \lambda (\xi-x)}+(4\epsilon a F)^{1/3}$, denoted by $\Gamma$, separates the points $(x_0,u(x_0))$ and $(x_2,u(x_2))$. On the other hand, by construction, the curve $\Gamma$ separates also the points $(0,u (0))$ and $(x_0,u(x_0))$.
This implies the existence of an interval $[x_1,x_2]$, with $0<x_1<x_0<x_2\leq \xi$, such that
\begin{itemize}
\item $(x_i,u(x_i))$ belongs to $\Gamma$, and $\big(u-(4\epsilon a F)^{1/3}\big)^2(x_i)=\lambda k(\xi-x_i)$, for $i=1,2$,
\item  $(x,u(x))$ is above $\Gamma$, and $\big(u-(4\epsilon a F)^{1/3}\big)^2(x)\geq \lambda k(\xi-x)$, for $x \in [x_1,x_2]$,
\item $u$ and also $\big(u-(4\epsilon a F)^{1/3}\big)^2$ are convex in $[x_1,x_2]$ 
\end{itemize}
which is clearly impossible. Thus, \eqref{claim1c} holds, and as a consequence $u$ is convex in $ [x_0,\xi]$. Now, let
$l:=\min \{x>\xi:\, u(x)=(4\epsilon a F)^{1/3}\}$. Thanks again to the convexity of $u$ in the region $D$, we see that $$u(x)\leq (4\epsilon a F)^{1/3}, \ \forall x\geq l.$$
In addition, $u$ is convex and decreasing in the interval $[x_0,l]$, since $u'(l)\leq 0$.
Our second claim is that $$\epsilon^2u''-\frac{u^3}{2}=\frac{u^3}{2}-\mu u -\epsilon a f\geq 0, \text{ on the interval $[x_0,l]$}.$$
This is true for $x \in [\xi,l]$, since $\frac{u^3}{2}\geq 2\epsilon a f$, and $-\mu u\geq 0$. We also check that when $x \in[x_0,\xi]$:
$$u^2\geq 4\mu+(4\epsilon a F)^{2/3} \text{ (by \eqref{claim1c})} \Rightarrow \frac{u}{2}(u^2-2\mu)\geq \mu u+(4\epsilon a F)^{2/3}\frac{u}{2}\geq 2\epsilon a F,$$
which establishes the second claim. Next, we obtain on the interval $[x_0,l]$:
$\epsilon^2u''u'-\frac{u^3u'}{2}\leq 0,$  which implies that the function
$[x_0,l]\ni x \to  4\epsilon^2|u'|^2-u^4$ is decreasing. Furthermore,
$4\epsilon^2|u'|^2-u^4\geq - u^4(l)=-(4\epsilon a F)^{4/3},$
and on the interval $[x_0,l)$ we have:
$$4\epsilon^2|u'|^2\geq u^4-(4\epsilon a F)^{4/3}\geq (u-(4\epsilon a F)^{1/3})^4 \Rightarrow\frac{-u'}{(u-(4\epsilon a F)^{1/3})^2}\geq\frac{1}{2\epsilon}.$$
An integration of the latter inequality over the interval $[x_0,\xi-\epsilon^{2/3}]$ gives:
$\frac{\epsilon^{1/3}}{u(\xi-\epsilon^{2/3})-(4\epsilon a F)^{1/3}}\geq\frac{(\xi-\epsilon^{2/3}-x_0)}{2\epsilon^{2/3}}$, 
and since 
$\frac{u(\xi-\epsilon^{2/3})-(4\epsilon a F)^{1/3}}{\epsilon^{1/3}}>\sqrt{ \frac{4\mu(\xi-\epsilon^{2/3})}{\epsilon^{2/3}}}\geq K''>0$, by \eqref{claim1c},
we deduce that $\xi-x_0\leq K'\epsilon^{2/3}$, with $K'=1+\frac{2}{K''}$. As a consequence, we have proved the upper bounds:
\begin{equation}
u(x)\leq 
\begin{cases}
\sqrt{8\lambda \mu(x)}+(4\epsilon a F)^{1/3} &\text{ for } x \in [0,\xi-K'\epsilon^{2/3}],\medskip \\
\sqrt{8\lambda \mu(\xi-K'\epsilon^{2/3})}+(4\epsilon a F)^{1/3} &\text{ for } x \in [\xi-K'\epsilon^{2/3},\infty).
\end{cases}
\end{equation}
The proof of the upper bound for $x\leq0$ is similar and simpler, since instead of $D$, we can consider the region $$D':=\left\{(x,y)\in (-\infty,0]\times [0,\infty): \, y>\sqrt{4\max(\mu(x),0)}\right\},$$ where the solutions are strictly convex. Finally, the lower bound follows from the odd symmetry $\hat u(x)=-u(-x)$. This completes the proof of \eqref{boundd}. The uniform bounds for $\tilde u^\pm$ are straightforward.
\end{proof}
\begin{step}(Proof of Theorems \ref{th2n} and \ref{th3n})
\end{step}
\begin{proof} We rescale the global minimizers $v$ as in Lemma \ref{l2} by setting $\tilde v_{\epsilon,a}^\pm(s)=\pm\frac{v_{\epsilon,a}(\pm\xi\pm s\epsilon^{2/3})}{\epsilon^{1/3}}$.
Without loss of generality we consider them only in a neighborhood of $\xi$, and write $\tilde v:=\tilde v_{\epsilon,a}^+$.
Clearly $\tilde v''(s)=\epsilon v''(\xi+s\epsilon^{2/3})$, thus,
\begin{equation}\label{oderes1}
\tilde v''(s)+\frac{\mu(\xi+s\epsilon^{2/3})}{\epsilon^{2/3}} \tilde v(s)-\tilde v^3(s)+ a f(\xi+s\epsilon^{2/3})=0, \qquad \forall s\in \R.\nonumber
\end{equation}
Writing $\mu(\xi+h)=\mu_1 h+hA(h)$, with $\mu_1:=\mu'(\xi)<0$, $A \in C(\R)$, and $A(0)=0$, we obtain
\begin{equation}\label{oderes2}
\tilde v''(s)+(\mu_1  + A(s \epsilon^{2/3}))s \tilde v(s)-\tilde v^3(s)+ a f(\xi+s\epsilon^{2/3})=0,\qquad  \forall s\in \R.
\end{equation}
Next, we define the rescaled energy by
\begin{equation}
\label{functres2}
\tilde E(\tilde u)=\int_{\R}\Big(\frac{1}{2}|\tilde u'(s)|^2-\frac{\mu(\xi+s \epsilon^{2/3})}{2\epsilon^{2/3}}\tilde u^2(s)+\frac{1}{4}|\tilde u|^4(s)- a f(\xi+s \epsilon^{2/3})\tilde u(s)\Big)\dd s.
\end{equation}
With this definition $\tilde E(\tilde u)=\frac{1}{\epsilon}E(u)$.
From Lemma \ref{l2} and \eqref{oderes2}, it follows that $\tilde v''$, and also $\tilde v'$, are uniformly bounded on compact intervals. Thanks to these uniform bounds, we can reproduce the arguments in the proof of Theorem \ref{th1n}, to obtain the convergence of $\tilde v_{\epsilon}$ to a minimal solution solution $\tilde V $ of the O.D.E. 
\begin{equation}\label{oderes4}
\tilde V''(s)+\mu_1 s \tilde V(s)-\tilde V^3(s)+ a f(\xi)=0, \qquad \forall s\in \R,
\end{equation}
which is associated to the functional
\begin{equation}
\label{functres4}
\tilde E_0(\phi,J)=\int_{J}\Big(\frac{1}{2}|\phi'(s)|^2-\frac{\mu_1}{2} s \phi^2(s)+\frac{1}{4}\phi^4(s)- a f(\xi)\phi(s) \Big)\dd s.
\end{equation}
Setting $y(s):=\frac{1}{\sqrt{2}(-\mu_1)^{1/3}}\tilde V\big(\frac{s}{(-\mu_1)^{1/3}}\big)$, \eqref{oderes4} reduces to \eqref{pain} with $\alpha=\frac{af(\xi)}{\sqrt{2}\mu_1}$, and $y$ is still a minimal solution of \eqref{pain} bounded at $\infty$. 
By taking global minimizers $v$ nonnegative on $[0,\infty)$, it is clear that at the limit we obtain $\tilde V\geq 0$, and $y\geq 0$. Lemmas \ref{pp0}, \ref{pp0bb} and \ref{pp0b} whose proofs are postponed for now, show that actually $y$ is positive, strictly decreasing, and has the asymptotic behavior described in Theorem \ref{th3n} (i) (cf. \eqref{asy0} and \eqref{asy1}). Finally, if we take global minimizers $v$ nonpositive on $(-\infty,0]$, we know by Theorem \ref{th1n} (iv) that their zero $\bar x$ converges to $\xi$, as $\epsilon \to 0$. However, we are not aware if their limit $\tilde V$ also vanishes. If so, the minimal solution $y$ has a unique zero $\bar s$, and behaves asymptotically as in \eqref{asy2} (cf. Lemma \ref{pp0b}). Note that proving that $\tilde V$ vanishes is actually equivalent to establishing the bound $|\bar x_\ve-\xi|=O(\ve^{2/3})$. The proof of the theorems is complete except for the Lemmas describing the asymptotic behaviour of the solutions of the Painlev\'e equation. 
\end{proof}

\section{Some Lemmas for solutions of the O.D.E. \eqref{pain}}
In this section we show Lemmas \ref{pp0}, \ref{pp0bb} and \ref{pp0b} announced above.  We begin with:

\begin{lemma}\label{roots}
Let us consider, for $\alpha<0$, the cubic equation
\begin{equation}\label{cubic}
2y^3+ s y+\alpha=0, \qquad \forall s \in \R,
\end{equation}
and let $ s^*:=-6|\frac{\alpha}{4}|^{2/3}<0$. Then
\begin{itemize}
\item for $s> s^*$, \eqref{cubic} has a unique real root $\sigma_+(s)$, which is positive;
\item for $s=s^*$, \eqref{cubic} has a simple zero $\sigma_+(s^*)>0$, and a double zero $\sigma_-(s^*)=\sigma_0(s^*)=-|\frac{\alpha}{4}|^{1/3}<0$;
\item for $s< s^*$, \eqref{cubic} has three simple zeros: $\sigma_+(s)>0$, and $\sigma_-(s)<\sigma_0(s)<0$.
\end{itemize}
Moreover,
\begin{itemize}
\item[(i)] $\sigma_+'(s)<0$, $\forall s \in \R$;
\item[(ii)] $\sigma_+(s)<\frac{|\alpha|}{ s}$, for $s>0$, and $\sigma_+\sim\frac{|\alpha|}{ s}$ at $+\infty$; 
\item[(iii)] $\sigma_+(s)>\sqrt{| s|/2}$, for $s<0$, and $\sigma_+(s)=\sqrt{| s|/2}+o(1)$, at $-\infty$;
\item[(iv)] $\sigma_+$ is convex in $[0,\infty)$, and concave in a neighborhood of $-\infty$. 
\end{itemize}
Similarly,
\begin{itemize}
\item[(v)] the function $(-\infty,s^*] \ni s \to \sigma_-(s)$ is strictly increasing;
\item[(vi)] $\sigma_-(s)>-\sqrt{|s|/2}$, for $s\leq  s^*$, and $\sigma_-(s)=-\sqrt{| s|/2}+o(1)$, at $-\infty$;
\item[(vii)] $\sigma_-$ is convex in a neighborhood of $-\infty$. 
\item[(viii)] $\sigma_0(s) \to 0$ as $s \to -\infty$. 
\item[(ix)] $\sigma_0$ is decreasing and concave in a neighborhood of $-\infty$. 
\end{itemize}
\end{lemma}
\begin{proof}
The first statement of the Proposition follows by studying the variations and the extrema of the polynomial in \eqref{cubic}. Let us prove the properties of $\sigma_+$. (i) By the implicit function theorem, it follows that $\sigma_+$ is differentiable. A computation shows that 
\begin{equation}\label{signsigma}
\sigma'_+=-\frac{1}{4\sigma_+-\frac{\alpha}{\sigma_+^2}}<0.
\end{equation}
Next, we notice that $2y^3+s y+\alpha >s y +\alpha \geq 0$, for $y\geq\frac{|\alpha|}{ s}$, with $s>0$, and this proves the inequality in (ii). Writing 
$\sigma_+(s)=\frac{|\alpha| }{2\sigma_+^2(s)+ s}$, we also obtain the equivalence in (ii). To see (iii), it is obvious that $2y^3+ s y+\alpha<0$, for $y=\sqrt{|s|/2}$, $s<0$. Thus, $\sigma_+(s)>\sqrt{|s|/2}$, for $s<0$. In addition, 
$$\sigma_+(s)-\sqrt{|s|/2}=\frac{|\alpha| }{2\sigma_+(s)(\sigma^+(s)+\sqrt{|s|/2})}=o(1).$$ (iv) Finally, we utilize again \eqref{signsigma}. Setting $\psi(s)=4\sigma_+(s)-\frac{\alpha }{\sigma_+^2(s)}$, we have $\psi'(s)=2\sigma'_+(s)\Big(2+\frac{\alpha}{\sigma^3_+(s)}\Big)<0$, as $s\to -\infty$, and $\psi'(s)>0$ for $s >0$. As a consequence, $\sigma'_+$ is decreasing (respectively increasing) in a neighborhood of $-\infty$ (resp. in $[0,\infty)$), and $\sigma_+$ is concave (resp. convex) in this neighborhood. The properties of $\sigma_-$ and $\sigma_0$ are established in a similar way.
\end{proof}

\begin{lemma}\label{pp0}
Let $\alpha<0$, and let $y$ be a solution of \eqref{pain}, bounded in a neighborhood of $+\infty$. Then, 
\begin{itemize}
\item[(i)] $y\geq \sigma_+$ in a neighborhood of $\infty$, 
\item[(ii)] $\sigma_-\leq y\leq \sigma_+$ in a neighborhood of $-\infty$, 
\item[(iii)] the function 
\begin{align}
\label{def-theta}
\theta (s)=|y'(s)|^2- s y^2(s)-y^4(s)-2 \alpha y(s) , \quad s\in\R
\end{align}
is decreasing, and converges to $0$ at $+\infty$.
\item[(iv)] $ y \sim \frac{|\alpha|}{s}$, as $s\to+\infty$.
\end{itemize} 
\end{lemma}
\begin{proof}
(i) Our first claim is that there exists a sequence $s_n \to +\infty$ such that $y(s_n)\geq \sigma_+(s_n)$ Assume by contradiction that this is not true.
Then, $y< \sigma_+$ on some interval $[m,\infty)$, where $y $ is also concave. Since $y$ is bounded on $[m,\infty)$, we deduce that $\lim_{+\infty} y'=0$, and $y' \geq 0$ on $[m,\infty)$. 
Furthermore, $\lim_{+\infty} y $ exists, and $ y  < 0$ on $[m,\infty)$. Now, we notice that by \eqref{pain}, $\theta'(s)=-y ^2\leq 0$, and thus $\theta$ is decreasing. This implies in particular that 
$\lim_{s\to+\infty}s y ^2(s)=l \in [0,\infty]$. If $l \neq 0$, it follows from \eqref{pain} that $\lim_{s\to+\infty} y ''(s)=-\infty$, which is impossible, since $ y $ is bounded in a neighborhood of $+\infty$.
Therefore, $\lim_{s\to+\infty}s y ^2(s)=0$, and $\theta(s) \geq 0$, $\forall s \in \R$. As a consequence, we have
$ | y '(s)|^2 \geq sy ^2(s)$, and $ -\frac{ y '}{ y }\geq \sqrt{ s}$ for $s>0$.
Integrating this inequality, we obtain that $ y (s)=\mathcal O(e^{-\frac{2}{3}s^{3/2}})$ at $+\infty$. By \eqref{pain} again, we conclude that 
$\lim_{s\to+\infty} y ''(s)=\alpha $, which contradicts the fact that $ y $ is bounded at $+\infty$. This establishes our first claim. To finish the proof of (i), let us assume that $ y (t)<\sigma_+(t)$, for some $t>s_k$, with $s_k$ such that $\sigma_+$ is convex on $[s_k,\infty)$. It follows that there exists an interval $[a,b]$ such that
\begin{itemize}
\item $s_k\leq a<t<b \leq s_l$ (for some $l>k$), 
\item $ y (a)=\sigma_+(a)$, $ y (b)=\sigma_+(b)$, and $ y (s)<\sigma_+(s)$, $\forall s \in (a,b)$.
\end{itemize}
Clearly, this is impossible since $\sigma_+- y $ is convex on $[a,b]$. Thus, we have proved that $ y \geq \sigma_+$ in a neighborhood of $+\infty$, where $ y $ is also convex.
Furthermore, by repeating the previous arguments, we obtain that $\lim_{+\infty} y '=0$ and $\lim_{s\to+\infty}s y ^2(s)=0$. 
Then, (iii) follows immediately. 

(ii)  We proceed as in (i). To show that $ y \leq \sigma_+$ in a neighborhood of $-\infty$, we first establish the existence of a sequence $s_n \to -\infty$ such that $ y (s_n)\leq \sigma_+(s_n)$. Assume by contradiction that this is not true.
Then, $ y >\sigma_+$ on some interval $(-\infty,m]$, where $ y $ is also convex. In addition, $ y '(s)<0$, $\forall s \leq m$, since otherwise $ y $ would be convex on all $\R$, and $\lim_{+\infty} y =+\infty$. As a consequence, there exists $m'<m$, such that $ y ^3(s)+2s  y (s)+4\alpha \geq 0$, $\forall s \leq m'$. Indeed, the positive root of the polynomial $ y ^3(s)+2s  y (s)+4\alpha $ is of order $\mathcal O(\sqrt{|s|})$ at $-\infty$. Next, in view of (iii), we obtain
$| y '(s)|^2-\frac{ y ^4(s)}{2}\geq \frac{ y (s)}{2}( y ^3(s)+2s y (s)+4\alpha )\geq 0$, $\forall s\leq m'$. An integration of the inequality $-\frac{ y '}{ y ^2}\geq\frac{1}{\sqrt{2}}$ over the interval $[s,m']$ gives
$\frac{1}{ y (m')}\geq\frac{1}{ y (m')}-\frac{1}{ y (s)}\geq\frac{m'-s}{\sqrt{2}}$, and letting $s\to -\infty$, we obtain a contradiction. This proves the existence of the sequence $s_n$. To deduce that $ y \leq \sigma_+$ in a neighborhood of $-\infty$, just repeat the convexity argument in (i). Finally, the proof of the bound $ y \geq \sigma_-$ is identical.

(iv) Let $\lambda>1$ be fixed, let $[m,\infty)$ be an interval where $ y $  is convex, and suppose there exists a sequence $m <s_k\to\infty$ such that $ y (s_k)>\lambda^2\frac{|\alpha |}{s_k}$. We notice that the inequality $\lambda^2\frac{|\alpha |}{ s_k}\geq \lambda\frac{|\alpha|}{s}$ holds for $s \geq\frac{s_k}{\lambda}$. Since $ y $ is decreasing on $[m,\infty)$, it follows that $ y (s)\geq \lambda\frac{|\alpha |}{s}$ for $s \in \big[\max \big(m,\frac{s_k}{\lambda}\big), s_k\big]$. In particular, by Lemma \ref{roots} (ii), we obtain on each interval $\big[\max \big(m,\frac{s_k}{\lambda}\big), s_k\big]$:
$$  2 y^3(s)+ s  y (s)+\lambda\alpha >0\Leftrightarrow  2y ^3(s)+ s  y (s)+\alpha>(\lambda-1)|\alpha | $$ 
since the positive root of the cubic equation $ 2y^3+s y+\lambda\alpha =0$ is smaller than $\lambda\frac{|\alpha |}{s}$. As a consequence $\int_m^\infty y ''(s)\dd s=\int_m^\infty ( 2y ^3(s)+s  y (s)+\alpha)\dd s=\infty$, which is a contradiction. Thus, we have proved  that for every $\lambda >1$, there exists a neighborhood of $+\infty$ where $\sigma^+\leq y \leq  \lambda^2\frac{|\alpha |}{s}$. This implies that  $ y \sim \frac{|\alpha |}{s}$, as $s\to+\infty$.
\end{proof} 

\begin{lemma}\label{pp0bb}
Let $\alpha=0$, and let $ y \geq 0$ be a minimal solution of \eqref{pain}, bounded at $\infty$. Then, $y$ coincides with the solution described in Theorem \ref{th3n} (i): it is positive, strictly decreasing, and satisfies \eqref{asy0}.

\end{lemma}
\begin{proof}
Let us show that $y>0$. If $y(s_0)=0$ for some $s_0 \in \R$, then $y$ has a local minimum at $s_0$, and $y \equiv 0$ by the uniqueness result for O.D.E. But this is excluded since a solution of \eqref{pain} which is bounded in a neighborhood of $-\infty$, is not minimal. To see this, we recall that for a minimal solution $y$, the second variation of the energy is nonnegative:
\begin{equation}\label{secondvar}
\int_\R (|\phi'(s)|^2+(6 y ^2(s)+s)\phi^2(s))\dd s \geq 0, \forall \phi\in C^1_0(\R),
\end{equation}
Clearly \eqref{secondvar} does not hold when $y$ is bounded and we take $\phi(s)=\phi_0(s+h)$, with $h\to\infty$, and $\phi_0\in C^1_0(\R)$ fixed. 
We also notice that $\lim_{s \to \infty}y'(s)=0$, and $y'(s)\leq 0$, $\forall s \geq 0$, since $y$ is convex and bounded on $[0,\infty)$. To obtain the asymptotic convergence at $+\infty$, we establish, as in Lemma \ref{pp0} (iii), 
that the function $\theta (s)=|y'(s)|^2- s y^2(s)-y^4(s)$ is decreasing, and converges to $0$ at $+\infty$.
As a consequence, $-\frac{y'}{y}\geq \sqrt{ s}, \ \forall s \geq 0$, and thus $y(s)\leq y(0)e^{-\frac{2}{3}s^{3/2}}, \ \forall s \geq 0 $.
Now, we refer to \cite{MR555581} where a complete classification of the solutions of \eqref{pain} converging to $0$ at $+\infty$ is established. 
It is known that among these solutions, only the one described in Theorem \ref{th3n} (i) does not converge to $0$ at $-\infty$. Clearly, $y$ does not converge to $0$ at $-\infty$, since it is not bounded by minimality, thus $y$ coincides with the aforementioned solution.
\end{proof}

\begin{lemma}\label{pp0b}
Let $\alpha<0$, and let $ y $ be a solution of \eqref{pain}, bounded at $\infty$. Then, 
\begin{itemize}
\item[(i)] if $ y  \geq 0$, we have $ y >0$, $ y '<0$, and $ y (s)=\sqrt{|s|/2}+o(1)$, as $s\to-\infty$.
\item[(ii)] if $ y  $ is minimal and vanishes at $\bar s$, we have $ y (s)>0 \Leftrightarrow s>\bar s$, $ y (s)<0 \Leftrightarrow s<\bar s$, and $ y (s)=-\sqrt{|s|/2}+o(1)$, as $s\to-\infty$.
\end{itemize} 
\end{lemma}
\begin{proof}
(i)  If $ y (s_0)=0$ for some $s_0 \in \R$, then $ y ''(s_0)\geq 0$, in contradiction with \eqref{pain} that gives  $ y ''(s_0)=\alpha <0$. Thus, $ y >0$. To show that $ y '<0$, we notice, that $ y (s)\geq\sigma_+(s)\Rightarrow  y '(s)<0$. Indeed, if $ y (s)\geq\sigma_+(s)$, and $ y '(s)\geq 0$, then $ y $ would be strictly convex in the interval $(s,+\infty)$, since $\sigma_+'<0$, and this would contradict the boundedness of $ y $ in $[s,+\infty)$. Similarly, we have that $0< y (s)<\sigma_+(s)\Rightarrow  y '(s)<0$. Here again, $0< y (s)<\sigma_+(s)$, and $ y '(s)\geq 0$, imply that $ y $ is strictly concave in the interval $(-\infty,s]$, in contradiction with $ y >0$. Now, let $\lambda>0$ be fixed,  and suppose there exists a sequence $s_k\to-\infty$ such that $0\leq  y (s_k) <\sqrt{| s_k|/2}-\lambda$. Since $ y '$ is bounded (in view of the bound $0\leq y \leq \sigma_+$, and the concavity of $y$), we notice that $0\leq  y (s) \leq \sqrt{|s_k|/2}\leq \sqrt{|s|/2}$, for $s\in [s_k-l,s_k]$, with $l$ independent of $k$. In particular, by Lemma \ref{roots} (iii), we obtain on each interval $[s_k-l, s_k]$:
$$ 2 y ^3(s)+ s  y (s)+\alpha \leq \alpha. $$ 
As a consequence $\int^{s_1}_{-\infty} y ''(s)\dd s=\int^{s_1}_{-\infty} ( 2y ^3(s)+ s  y (s)+\alpha )\dd s=-\infty$, which is a contradiction. Thus, we have proved  that for every $\lambda >0$, there exists a neighborhood of $-\infty$ where $\sigma^+\geq y \geq  \sqrt{|s|/2}-\lambda$. This implies that  $ y = \sqrt{| s|/2}+o(1)$, as $s\to-\infty$.

(ii) If $ y $ is minimal and vanishes at $\bar s$, it is easy to see that this zero is unique. Indeed, if $ y $ also vanishes at $\bar s'<\bar s$, we have $ y \geq 0$ on $[\bar s',\bar s]$, since otherwise we would obtain $E_{\mathrm{P_{II}}}( y ,[\bar s',\bar s])> E_{\mathrm{P_{II}}}(| y |,[\bar s',\bar s])$. It follows that $\bar s$ is a local minimum of $ y $ in contradiction with \eqref{pain}. Another consequence of the minimality of $ y $, is the inequality \eqref{secondvar}, which implies that $ y $ is not bounded at $-\infty$ (cf. Lemma \ref{pp0bb}). Let $l<0$ be fixed, and let $s_k\to-\infty$ be a sequence such that $ y (s_k)<l$. We notice that $\min_{u \in [l,0]}\big(\frac{1}{2}u^4+\frac{s}{2}u^2+\alpha u\big)$ is attained for $y=l$, when $s<s_i$, with $|s_i|$ large enough. Thus, if $ y (s) >l$ for some $s<s_i$, we can find an interval $[a,b]$ containing $s$, such that $ y (a)= y (b)=l$, and $\tilde E_0( y ,[a,b])>\tilde E_0(\min( y ,l),[a,b])$, which is a contradiction. This proves that $ y (s) \leq l$ for $s<s_i$ i.e. $\lim_{-\infty} y =-\infty$. It also follows that $ y $ is convex in a neighborhood of $-\infty$, since $\sigma_-\leq y\leq \sigma_0$. Utilizing the convexity of $y$, one can establish as in (i) that $ y (s)=-\sqrt{|s|/2}+o(1)$, as $s\to-\infty$.
\end{proof}

\bigskip
{\noindent}
{\bf Acknowledgement.} We would like to thank William Troy and Stuart Hastings for  observations  that helped  us  to implement  some important improvements in the present version of this work.



\providecommand{\bysame}{\leavevmode\hbox to3em{\hrulefill}\thinspace}
\providecommand{\MR}{\relax\ifhmode\unskip\space\fi MR }
\providecommand{\MRhref}[2]{%
  \href{http://www.ams.org/mathscinet-getitem?mr=#1}{#2}
}
\providecommand{\href}[2]{#2}

\end{document}